\documentclass[11pt]{amsart}
\usepackage{amssymb}

\newtheorem{theorem}{Theorem}[section]
\newtheorem{corollary}[theorem]{Corollary}
\newtheorem{lemma}[theorem]{Lemma}
\newtheorem{proposition}[theorem]{Proposition}
\theoremstyle{definition}
\newtheorem{definition}[theorem]{Definition}

\newtheorem{question}[theorem]{Question}
\theoremstyle{remark}
\newtheorem{remark}[theorem]{Remark}
\newtheorem{example}[theorem]{Example}

\newtheorem{claim}{\bf Claim}

\numberwithin{equation}{section}

\def\R {{\mathbb{R}}}
\def\T {{\mathbb{T}}}
\def\Q {{\mathbb{Q}}}

\def\N{{\mathbb{N}}}

\def\i{{\rm i}}

\def\Z {{\mathbb{Z}}}
\def\diam{{\mathrm{diam}\,}}

\def\3{{|\!|\!|}}

\begin{document}

\title[On equivalence relations induced by LCA Polish groups]{On equivalence relations induced by locally compact abelian Polish groups}
\author{Longyun Ding}
\address{School of Mathematical Sciences and LPMC, Nankai University, Tianjin, 300071, P.R.China}
\email{dingly@nankai.edu.cn}
\thanks{Research is partially supported by the National Natural Science Foundation of China (Grant No. 11725103).}
\author{Yang Zheng}
\email{1120200015@mail.nankai.edu.cn}


\begin{abstract}
Given a Polish group $G$, let $E(G)$ be the right coset equivalence relation $G^\omega/c(G)$, where $c(G)$ is the group of all convergent sequences in $G$. The connected component of the identity of a Polish group $G$ is denoted by $G_0$.

Let $G,H$ be locally compact abelian Polish groups. If $E(G)\leq_B E(H)$, then there is a continuous homomorphism $S:G_0\rightarrow H_0$ such that $\ker(S)$ is non-archimedean. The converse is also true when $G$ is connected and compact.

For $n\in\N^+$, the partially ordered set $P(\omega)/\mbox{Fin}$ can be embedded into Borel equivalence relations between $E(\R^n)$ and $E(\T^n)$.
\end{abstract}
\maketitle

\section{Introduction}

A topological space is {\it Polish} if it is separable and completely metrizable. For more details in descriptive set theory, we refer to~\cite{KEchris}. It is an important application of descriptive set theory to study equivalence relations by using Borel reducibility.
Given two Borel equivalence relations $E$ and $F$ on Polish spaces $X$ and $Y$ respectively, recall that $E$ is {\it Borel reducible} to $F$, denoted $E\leq_B F$, if there exists a Borel map $\theta:X\rightarrow Y$ such that for all $x,y\in X$,
$$xEy\Longleftrightarrow\theta(x)F\theta(y).$$
We denote $E\sim_B F$ if both $E\leq_B F$ and $F\leq_B E$, and denote $E<_B F$ if $E\leq_B F$ and $F\nleq_B E$. We refer to~\cite{gao} for background on Borel reducibility.

Polish groups are important tools in the research on Borel reducibility. A topological group is {\it Polish} if its topology is Polish. For a Polish group $G$, the authors~\cite{DZ} defined an equivalence relation $E(G)$ on $G^\omega$ by
$$xE(G)y\iff\lim_nx(n)y(n)^{-1}\mbox{ converges in }G$$
for $x,y\in G^\omega$. We say that $E(G)$ is the {\it equivalence relation induced by} $G$. Indeed, $E(G)$ is the right coset equivalence relation $G^\omega/c(G)$, where $c(G)$ is the group of all convergent sequences in $G$.

In this article, we focus on equivalence relations induced by locally compact abelian Polish groups. Some interesting results have been found in some special cases. For instance, for $c_0,e_0,c_1,e_1\in\N$, $E(\R^{c_0}\times\T^{e_0})\le_BE(\R^{c_1}\times\T^{e_1})$ iff $e_0\le e_1$ and $c_0+e_0\le c_1+e_1$ (cf.~\cite[Theorem 6.19]{DZ}).

Given a group $G$, the identity element of $G$ is denoted by $1_G$. If $G$ is a topological group, the connected component of $1_G$ in $G$ is denoted by $G_0$. Recall that a Polish group $G$ is {\it non-archimedean} if it has a neighborhood base of $1_G$ consisting of open subgroups.

\begin{theorem}
Let $G$ and $H$ be two locally compact abelian Polish groups.
If $E(G)\leq_B E(H)$, then there is a continuous homomorphism $S:G_0\rightarrow H_0$ such that $\ker(S)$ is non-archimedean.
\end{theorem}

By restricting attention to compact connected abelian Polish groups, we prove the following theorem.

\begin{theorem}[Rigid Theorem]
Let $G$ be a compact connected abelian Polish group and $H$ a locally compact abelian Polish group. Then $E(G)\leq_B E(H)$ iff there is a continuous homomorphism $S:G\rightarrow H$ such that $\ker(S)$ is non-archimedean.
\end{theorem}

For every normal space $X$, denoted by $\dim(X)$ the {\it covering dimension} of $X$, where $\dim(X)$ is an integer $\ge-1$ or the ``infinite number $\infty$''. Let $G$ be an abelian topological group. The topological group Hom$(G,\T)$ is called the {\it dual group} of $G$, denoted by $\widehat{G}$~(see Section 4). For any finite dimensional compact abelian Polish group $G$, $\dim(G)={\rm rank}(\widehat{G})$, the torsion-free rank of $\widehat G$~(cf. Lemma 8.13 and Corollary 8.26 of~\cite{HM13}). We say $G$ is $n$-{\it dimensional} if $\dim(G)=n$ for some integer $n$, or {\it infinite dimensional} if $\dim(G)$ is infinite.

Recall that $\T$ is the multiplicative group of all complex numbers with modulus $1$. For finite dimensional compact abelian Polish groups, we obtain the following results.

\begin{theorem}
Let $G,H$ be locally compact abelian Polish groups.
 \begin{enumerate}
    \item If $G$ is non-archimedean, then $E(G)\leq_B E_0^\omega$.
    \item If $G$ is not non-archimedean, then $E(\R)\leq_B E(G)$.
    \item If $G$ is not non-archimedean and $G_0$ is open, then $E(G)\sim_B E(G_0)$.
    \item If $n$ is a positive integer, then $E(\T^n)\leq_B E(G)$ iff $\T^n$ embeds in $G$.
    \item If $n$ is a positive integer and $G$ is compact, then $G$ is $n$-dimensional iff $E(\R^n)<_B E(G)\leq_B E(\T^n)$.
  \end{enumerate}
\end{theorem}

Let $\mathcal{P}$ denote the set of all primes. For $P,Q\in\mathcal P^\omega$, $Q\preceq P$ means that there is a co-finite subset $A$ of $\omega$ and an injection $f:A\to\omega$ such that $Q(n)=P(f(n))$ for each $n\in A$.

For $P\in\mathcal{P}^\omega$, we consider the closed subgroup of $\T^\omega$, named $P$-{\it adic solenoid}, $\Sigma_P=\{g\in\T^\omega:\forall l\,(g(l)=g(l+1)^{P(l)})\}$ (cf.~\cite{Gu}).

\begin{theorem}
Let $P$ and $Q$ be in $\mathcal{P}^\omega$. Then $E(\Sigma_P)\leq_B E(\Sigma_Q)$ iff $Q\preceq P$.
\end{theorem}

The partially ordered set $P(\omega)/\mbox{Fin}$ is so complicated that every Boolean algebra of size $\le\omega_1$ embeds into it (see~\cite{BDHHMU}). We usually express that some classes of Borel equivalence relations are extremely complicated under the order of Borel reducibility by showing that $P(\omega)/\mbox{Fin}$ embeds into them. For instance, Louveau-Velickovic~\cite{LV} and Yin~\cite{yin} showed that $P(\omega)/\mbox{Fin}$ embeds into both LV-equalities and Borel equivalence relations between $\ell_p$ and $\ell_q$ respectively. As an application, we prove that, the partially ordered set $P(\omega)/\mbox{Fin}$ embeds into the partially ordered set of all $E(G)$'s under the ordering of Borel reducibility.

\begin{theorem}
Let $n\in\mathbb{N}^+$. Then for $A\subseteq\omega$, there is a $n$-dimensional compact connected abelian Polish group $G_A$ such that $E(\R^n)<_B E(G_A)<_B E(\T^n)$ and for $A,B\subseteq\omega$, we have
$$A\subseteq^*B\Longleftrightarrow E(G_A)\leq_B E(G_B).$$
\end{theorem}

We also get a sufficient and necessary condition concerning dual groups.

\begin{theorem}[Dual Rigid Theorem]
Let $G$ be a compact connected abelian Polish group and $H$ a locally compact abelian Polish group. Then $E(G)\leq_B E(H)$ iff there is a continuous homomorphism $S^*:\widehat{H}\to\widehat{G}$ such that $\widehat{G}/{\rm im}(S^*)$ is a torsion group.
\end{theorem}

{\bf Notation convention:} In this article, the addition operation of any subgroup of $\R^n$ is denoted by $+$ and its identity element is denoted by $0$. Unless otherwise specified, for abstract abelian topological groups $G$, we still use multiplicative notation to express the group operation, and use $1_G$ to express the identity element of $G$, since we often consider subgroups of $\T^\omega$.

This article is organized as follows. In section 2, we prove theorems 1.1--1.3. In section 3, we consider $P$-{adic solenoids} and prove theorems 1.4 and 1.5. Finally, In section 4, we consider dual groups and prove Theorem 1.6.

\section{Locally compact abelian Polish groups}

\begin{definition}[{\cite[Definition 6.1]{DZ}}]
Let $G$ be a Polish group. We define equivalence relation $E_*(G)$ on $G^\omega$ as: for $x,y\in G^\omega$,
$$xE_*(G)y\iff\lim_nx(0)x(1)\dots x(n)y(n)^{-1}\dots y(1)^{-1}y(0)^{-1}\mbox{ converges}.$$
\end{definition}

The following is an easy but important observation.

\begin{proposition}
Let $G$ be a Polish group. Then $E(G)\sim_B E_*(G)$.
\end{proposition}

\begin{proof}
To see that $E(G)\leq_B E_*(G)$, for $x\in G^\omega$, we define $\theta(x)\in G^\omega$ as
$$\theta(x)(n)=\left\{\begin{array}{ll}x(0), & n=0,\cr x(n-1)^{-1}x(n), & n>0.\end{array}\right.$$
Then $\theta$ witnesses that $E(G)\leq_B E_*(G)$.

To show the converse, for $x\in G^\omega$, we define $\vartheta(x)\in G^\omega$ as
$$\vartheta(x)(n)=x(0)x(1)\cdots x(n).$$
Then $\vartheta$ witnesses that $E_*(G)\leq_B E(G)$.
\end{proof}
In this article, we focus on abelian Polish groups. For abelian Polish groups $G$, it is more convenient to take $E_*(G)$ as research object than $E(G)$.

Some reducibility results are obtained in~\cite{DZ}. Since we will use them again and again in this article, for the convenience of readers, we list them as follows.

\begin{proposition}[{\cite[Proposition 3.4]{DZ}}]\label{closed subgroup}
Let $G,H$ be two Polish groups. If $G$ is topologically isomorphic to a closed subgroup of $H$, then $E(G)\le_BE(H)$.
\end{proposition}

A metric $d$ on a group $G$ is called {\it two sided invariant} if $d(ghl,gkl)=d(h,k)$ for all $g,h,k,l\in G$. We say that a Polish group $G$ is TSI if it admits a compatible two sided invariant metric. Any abelian Polish group is TSI~(cf.~\cite[Exercise 2.1.4]{gao}).

\begin{lemma}[{\cite[Theorem 6.5]{DZ}}]\label{[0,1-embed]}
Let $G,H,K$ be three TSI Polish groups. Suppose $\psi:G\to H$ and $\varphi:H\to K$ are continuous homomorphisms with $\varphi(\psi(G))=K$ such that $\ker(\varphi\circ\psi)$ is non-archimedean. If the interval $[0,1]$ embeds into $H$, then $E(G)\le_BE(H)$.
\end{lemma}

\begin{lemma}[{\cite[Theorem 6.13]{DZ}}]\label{borereduce}
Let $G,H$ are be TSI Polish groups such that $H$ is locally compact. If $E(G)\le_BE(H)$, then there exist an open normal subgroup $G_c$ of $G$ and a continuous map $S:G_c\to H$ with $S(1_G)=1_H$ such that, for $x,y\in G_c^\omega$, if $\lim_nx(n)y(n)^{-1}=1_G$, then
$$xE_*(G_c)y\iff S(x)E_*(H)S(y),$$
where $S(x)(n)=S(x(n)),S(y)(n)=S(y(n))$ for each $n\in\omega$.

In particular, if $G=G_c$ and the interval $[0,1]$ embeds in $H$, then the converse is also true.
\end{lemma}

\begin{remark}\label{G_0toH_0}
Since $G_c$ in the preceding lemma is an open subgroup, it is also closed. So $G_0\subseteq G_c$ as it is connected. Since $S$ is continuous, we have $S(G_0)\subseteq H_0$. Moreover, for all $x,y\in G_0^\omega$, if $\lim_nx(n)y(n)^{-1}=1_G$, we have
$$xE_*(G_0)y\Longleftrightarrow S(x)E_*(H_0)S(y).$$
\end{remark}

The next lemma plays the key role in the proof of Theorem~\ref{lca}.

\begin{lemma}\label{Base}
Let $G$ and $H$ be two abelian Polish groups such that:
\begin{enumerate}
\item [(1)] $H$ is locally compact,
\item [(2)] $H_0\subseteq\R^\omega\times\T^\omega$,
\item [(3)] there is a nonzero continuous homomorphism $f:\R^m\rightarrow G$ for some $m\in\N^+$.
\end{enumerate}
If $E_*(G)\leq_B E_*(H)$, then there is a continuous map $S:G_0\rightarrow H_0$ such that
the map $S$ restricted on $f(\R^m)$ is a homomorphism to $H_0$.
\end{lemma}

\begin{proof}
First, from Remark~\ref{G_0toH_0}, we can obtain a continuous map $S:G_0\rightarrow H_0$ with $S(1_{G_0})=1_{H_0}$ such that, for $x,y\in G_0^\omega$, if $\lim_nx(n)y(n)^{-1}=1_{G_0}$, then
$$xE_*(G_0)y\Longleftrightarrow S(x)E_*(H_0)S(y),$$
where $S(x)(n)=S(x(n)),S(y)(n)=S(y(n))$ for each $n\in\omega$.

Since $H_0\subseteq\R^\omega\times\T^\omega$, without loss of generality we may assume that $h(2k)\in\R$ and $h(2k+1)\in\T$ for all $h\in H_0.$ For $k\in\omega$, we define continuous homomorphims
$\phi^{2k}:H_0\rightarrow\R$ and $\phi^{2k+1}:H_0\rightarrow\T$ by $\phi^j(h)=h(j)$.

Now fix $g_0,g_1\in f(\R^m)$ and find $a_0,a_1\in\R^m$ such that $f(a_0)=g_0$ and $f(a_1)=g_1$.
For $t\in[0,1]$ and $l\in\{1,2\}$, define $a^l(t)\in\R^m$ as
$$a^l(t)=\left\{\begin{array}{ll}a_0+t(a_1-a_0), & l=1,\cr t(a_0+a_1), & l=2.\end{array}\right.$$
By the following claim, we can easily construct a continuous function $F^l_j:[0,1]\rightarrow\R$ for each $l\in\{1,2\}$ and $k\in\omega$ such that
$$F^l_{2k}(t)=\phi^{2k}(S(f(a^l(t)))),\quad \exp(\i F^l_{2k+1}(t))=\phi^{2k+1}(S(f(a^l(t)))).\leqno{(*)}$$
The nontrivial part of the construction, i.e., $j=2k+1$, follows from a more general claim.

\begin{claim}
Given a continuous function $\gamma:[0,1]\rightarrow\T$ and $t_0\in[0,1]$ with $\exp(\i s_0)=\gamma(t_0)$ for some $s_0\in\R$, there exists a continuous function $\widetilde{\gamma}:[0,1]\rightarrow\R$ such that $\exp(\i\widetilde{\gamma}(t))=\gamma(t)$ and $\widetilde{\gamma}(t_0)=s_0$.
\end{claim}

\begin{proof}
Note that the map $t\mapsto\exp(\i t)$ is a covering map from $\R$ to $\T$, and the interval $[0,1]$ is simply connected~(see Definitions A2.1 and Proposition A2.8 of~\cite{HM13}). So such a $\widetilde{\gamma}$ exists~(cf.~\cite[Definition A2.6]{HM13}).

For the convenience of readers, we briefly explain the construction of $\widetilde{\gamma}$. Since the map $t\mapsto\exp(it)$ is a local homeomorphism, by the continuity of $\gamma$, for each $u\in[0,1]$, there is an open interval $J_u$ containing $u$ and a continuous function $\widetilde{\gamma_u}:J_u\cap[0,1]\to\R$ such that $\sup_{t,t'\in J_u\cap[0,1]}|\gamma(t)-\gamma(t')|<\frac{1}{2}$ and  $\exp(\i\widetilde{\gamma_u}(t))=\gamma(t)$ for $t\in J_u\cap[0,1]$. Note that $\exp(\i(\widetilde{\gamma_u}(t)+2p\pi))=\exp(\i\widetilde{\gamma_u}(t))$ for each $p\in\Z$. By the compactness of $[0,1]$, there are $u_0,u_1,\cdots,u_q\in[0,1]$ such that $[0,1]\subseteq \bigcup_{0\leq i\leq q} J_{u_i}$. We can find $0=p_0,p_1,\cdots,p_q\in\mathbb{Z}$ such that for each $t\in J_{u_i}\cap J_{u_j}\cap[0,1]$, we have $\widetilde{\gamma_{u_i}}(t)+2p_i\pi=\widetilde{\gamma_{u_j}}(t)+2p_j\pi$. Then for $t\in[0,1]\cap J_{u_i}$, let $\widetilde\gamma'(t)=\widetilde{\gamma_{u_i}}(t)+2p_i\pi$. In the end, we put $\widetilde{\gamma}(t)=\widetilde\gamma'(t)-\widetilde\gamma'(t_0)+s_0$. It is obvious that $\exp(\i\widetilde{\gamma}(t))=\gamma(t)$ and $\widetilde{\gamma}(t_0)=s_0$.
\end{proof}

Note that $S(f(a^2(0)))=1_H$. We can assume that $F^2_j(0)=0$ for each $j$.

Next we claim that $F^l_j$ are linear functions.

\begin{claim}~\label{Cla2}
$F^l_j(t)=F^l_j(0)+t(F^l_j(1)-F^l_j(0))$ for $t\in[0,1]$.
\end{claim}

\begin{proof}
We only verify the claim for $l=1$. It is similar for $l=2$.

Fix $j_0\in\omega$. Define $\gamma:[0,1]\to\R$ as $\gamma(t)=F^1_{j_0}(t)-F^1_{j_0}(0)-t(F^1_{j_0}(1)-F^1_{j_0}(0))$. Note that $\gamma$ is continuous and $\gamma(0)=\gamma(1)=0$. We only need to prove that $\gamma(t)=0$ for all $t\in(0,1)$.

If not, without loss of generality we may assume that $\gamma(t_0)>0$ for some $t_0\in(0,1)$. Similar to the proof of~\cite[Lemma 6.17]{DZ}, we can find $0<\xi_0<\xi_1<\xi_2<\cdots<\xi<1$ such that $\gamma(\xi_k)=\frac{k+1}{k+2}\gamma(t_0)$ for each $k\in\omega$, and $1>\zeta_0>\zeta_1>\zeta_2>\cdots>\zeta>0, K\in\omega$ such that, for $k\ge K$, we have
$$\xi-\xi_k>\zeta_k-\zeta>\xi-\xi_{k+1},$$
$\lim_k \xi_k=\xi$, $\lim_k \zeta_k=\zeta$, $\gamma(\xi)=\gamma(t_0)$, and $\gamma(\zeta)>\gamma(\zeta_k)$ for each $k$.

Note that $f:\R^m\rightarrow G$ is a nonzero continuous homomorphism. For $p\in\omega$, we set
$$x(p)=\left\{\begin{array}{ll}f(a^1(\xi)), & p=2k,\cr f(a^1(\zeta)), & p=2k+1,\end{array}\right.\quad
y(p)=\left\{\begin{array}{ll}f(a^1(\xi_k)), & p=2k,\cr f(a^1(\zeta_k)), & p=2k+1.\end{array}\right.$$

From the alternating series test, the following series
$$(\xi-\xi_0)+(\zeta-\zeta_0)+\cdots+(\xi-\xi_k)+(\zeta-\zeta_k)+\cdots$$
is convergent.
Then
$$\begin{array}{ll}
&x(0)x(1)\cdots x(2k)y(2k)^{-1}\cdots y(1)^{-1}y(0)^{-1} \cr
=&x(0)y(0)^{-1}x(1)y(1)^{-1}\cdots x(2k)y(2k)^{-1}\cr
=&f(((\xi-\xi_0)+(\zeta-\zeta_0)+\cdots+(\xi-\xi_k))(a_1-a_0)).\cr
\end{array}$$
Since $f$ is continuous and $\lim_p x(p)y(p)^{-1}=1_G$, we have $xE_*(G)y$. And hence, by Remark~\ref{G_0toH_0}, we have $S(x)E_*(H)S(y)$.

On the other hand, we have
$$\sum_k(\gamma(\xi)-\gamma(\xi_k)+\gamma(\zeta)-\gamma(\zeta_k))\ge\sum_k(\gamma(\xi)-\gamma(\xi_k))=\sum_k\frac{\gamma(t_0)}{k+2}=\infty.$$
Note that
$$\begin{array}{ll}
&F^1_{j_0}(\xi)-F^1_{j_0}(\xi_k)+F^1_{j_0}(\zeta)-F^1_{j_0}(\zeta_k)\cr
=&\gamma(\xi)-\gamma(\xi_k)+\gamma(\zeta)-\gamma(\zeta_k)+(\xi-\xi_k+\zeta-\zeta_k)(F^1_{j_0}(1)-F^1_{j_0}(0)).\cr
\end{array}$$
If $j_0=2i$, then
$$\begin{array}{ll}
&\phi^{2i}(S(x(0))S(x(1))\cdots S(x(2k))S(y(2k))^{-1}\cdots S(y(1))^{-1}S(y(0))^{-1})\cr
=&F^1_{j_0}(\xi)-F^1_{j_0}(\xi_0)+F^1_{j_0}(\zeta)-F^1_{j_0}(\zeta_0)+\cdots+F^1_{j_0}(\xi)-F^1_{j_0}(\xi_k).\cr
\end{array}$$
Thus $S(x)E_*(H)S(y)$ fails. We get a contradiction. If $j_0=2i+1$, following similar arguments, we can also get a contradiction. This complete the proof of the claim.
\end{proof}

Now by Claim~\ref{Cla2} and $F^2_j(0)=0$, we know that
$$F^1_j(1/2)=F^1_j(0)+(F^1_j(1)-F^1_j(0))/2=(F^1_j(0)+F^1_j(1))/2,$$
$$F^2_j(1/2)=F^2_j(1)/2.$$
By comparing equation $(*)$ before Claim 1, it follows that
$$S(f(a^1(1/2)))^2=S(f(a_0))S(f(a_1))=S(g_0)S(g_1),$$
$$S(f(a^2(1/2)))^2=S(f(a_0+a_1))=S(g_0g_1).$$
Since $a^1(1/2)=a^2(1/2)$, we have $S(g_0)S(g_1)=S(g_0g_1)$.

So, the map $S:f(\R^m)\rightarrow H_0$ is a continuous homomorphism.
\end{proof}


Let us recall the structure of Hausdorff locally compact abelian groups. Let $G$ be a Hausdorff locally compact abelian group, then $G$ is topologically isomorphic to the group $\R^n\times H$, where $H$ is a locally compact abelian group containing a compact open subgroup~(cf.~\cite[Theorem 24.30]{HAR}). Moreover, if $G$ is also connected, then it is a direct product of a compact connected abelian group $K$ and the group $\R^n$~(cf.~\cite[Theorem 9.14]{HAR}). Any locally compact connected metrizable abelian group can be embedded as a closed subgroup of $\R^n\times\T^\omega$. In particular, all compact metrizable abelian groups can be embedded in $\T^\omega$~(see Page 119 of~\cite{Ar}). $G$ is said to be {\it solenoidal} if there is a continuous homomorphism $f:$ $\R\rightarrow G$ such that $f(\R)$ is dense in $G$~(see~\cite[(9.2)]{HAR}). It is well known that a compact metrizable abelian group is solenoidal iff it is connected~(see Page 13 and Proposition 5.16 of~\cite{Ar}). Thus for each locally compact connected metrizable abelian group $G$, there is a continuous homomorphism $f:$ $\R^m \rightarrow G$ which satisfies $\overline{f(\R^m)}=G$. For more details on locally compact abelian groups, we refer to~\cite{Ar,HAR}.

By applying Lemma~\ref{Base} for locally compact abelian Polish groups, we get the following results.

\begin{theorem}\label{lca}
Let $G$ and $H$ be two locally compact abelian Polish groups.
If $E(G)\leq_B E(H)$, then there is a continuous homomorphism $S:G_0\rightarrow H_0$ such that $\ker(S)$ is non-archimedean.
\end{theorem}

\begin{proof}
If $E(G)\leq_B E(H)$, then $E_*(G)\leq_B E_*(H)$. Without loss of generality we may assume that $G_0$ is nontrivial.
First note that $H_0$ can be embedded into $\R^n\times\T^\omega$.  Thus we may assume without loss of generality that $H_0\subseteq\R^\omega\times\T^\omega$. Let $f$ be a continuous homomorphism from $\R^m$ to $G_0$ with $\overline{f(\R^m)}=G_0$. Then by Lemma \ref{Base} there exists a continuous map $S:G_0\rightarrow H_0$ such that the map $S$ restricted on $f(\R^m)$ is a homomorphism to $H_0$. Since $f(\R^m)$ is dense in $G_0$, we see that $S$ is a homomorphism from $G_0$ to $H_0$.

Then we only need to check that $\ker(S)$ is non-archimedean. Assume toward a contradiction that $\ker(S)$ is not non-archimedean.

Note that $\ker(S)$ is an abelian Polish group. Fix a compatible two sided invariant metric on $\ker(S)$. Let $V_k\subseteq\ker(S),\,k\in\omega$ be an open symmetric neighborhood base of $1_{\ker(S)}=1_G$ with $\lim_k\diam(V_k)=0$. Then there exists a $k_0\in\omega$ such that $V_{k_0}$ does not contain any open subgroup of $\ker(S)$. Since $V_k$ is symmetric, $\bigcup_m V_k^m$ is an open subgroup of $\ker(S)$, so $\bigcup_m V_k^m\nsubseteq V_{k_0}$ for each $k$. Thus we can find an $m_k\in\omega$ and $g_{k,0},\dots,g_{k,m_k-1}\in V_k$ such that
$g_{k,0}g_{k,1}\cdots g_{k,m_k-1}\notin V_{k_0}$.

Denote $M_{-1}=0$ and $M_k=m_0+m_1+\dots+m_k$ for $k\in\omega$. Now for $n\in\omega$, define
$$x(n)=\left\{\begin{array}{ll}g_{k,j}, &n=M_{k-1}+j,0\leq j<m_k,\cr 1_G, & \mbox{otherwise.}\end{array}\right.$$
Therefore $xE_*(G)1_{G^\omega}$ fails.
Note that we have $\lim_nx(n)=1_G$ and $S(x(n))=1_H$ for each $n$. So it is trivial that $S(x)E_*(H_0)S(1_{G^\omega})$, where $S(x)(n)=S(x(n))$, contradicting Lemma~\ref{borereduce}.
\end{proof}

In paricular, if $G$ is compact connected, then the converse of Theorem~\ref{lca} is also true.

\begin{theorem}[Rigid Theorem]\label{cc}
Let $G$ be a compact connected abelian Polish group and $H$ a locally compact abelian Polish group. Then $E(G)\leq_B E(H)$ iff there is a continuous homomorphism $S:G\rightarrow H$ such that $\ker(S)$ is non-archimedean.
\end{theorem}

\begin{proof}
Let $S$ be a continuous homomorphism from $G$ to $H$ such that $\ker(S)$ is non-archimedean. Since $G$ is compact, $S(G)$ is a compact, thus closed subgroup of $H$. So we have $E(S(G))\le_BE(H)$.

Note that $S(G)$ is also a compact connected abelian Polish group. Let $f$ be a continuous homomorphism $f:\R\rightarrow S(G)$ such that $\overline{f(\R)}=S(G)$. Then $\ker(f)$ is a proper closed subgroup of $\R$. Hence $\ker(f)$ is a discrete group. This gives that the interval $[0,1]$ embeds in $S(G)$. Then by Lemma~\ref{[0,1-embed]}, we get that $E(G)\leq_B E(S(G))\le_BE(H)$.

On the other hand, if $E(G)\le_BE(H)$, by Theorem~\ref{lca}, there is a continuous homomorphism $S:G_0\to H_0$ such that $\ker(S)$ is non-archimedean. Since $G$ is connected, we have $G=G_0$.
\end{proof}

\begin{corollary}
Let $G$ be a compact connected abelian Polish group and $H$ a locally compact abelian Polish group. Suppose $H_0\cong\R^n\times K$, where $K$ is a compact connected abelian group. Then $E(G)\le_BE(H)$ iff $E(G)\le_BE(K)$.
\end{corollary}

\begin{proof}
$(\Leftarrow)$ part is trivial, since $E(K)\le_BE(H_0)\le_BE(H)$.

$(\Rightarrow)$. If $E(G)\le_BE(H)$, then there exists a continuous homomorphism $S:G\to H$ such that $\ker(S)$ is non-archimedean. So $S(G)$ is a connected compact subgroup of $H$, thus $S(G)\subseteq H_0$. Without loss of generality, we assume that $H_0=\R^n\times K$. Let $\pi:H_0\to\R^n$ and $\pi_K:H_0\to K$ be canonical projections. Then $\pi(S(G))$ is a compact subgroup of $\R^n$, so $\pi(S(G))=\{0\}$. It follows that $\ker(\pi_K\circ S)=\ker(S)$. Applying Theorem~\ref{cc} on $\pi_K\circ S:G\to K$, we have $E(G)\le_BE(K)$.
\end{proof}

Recall that a topological group $G$ is {\it totally disconnected} if $G_0=\{1_G\}$. For any locally compact abelian Polish group, it is totally disconnected iff it is non-archimedean (cf.~\cite[Theorem 1.34]{HM13}).

For every normal space $X$, denoted by $\dim(X)$ the {\it covering dimension} of $X$, where $\dim(X)$ is an integer $\ge-1$ or the ``infinite number $\infty$''. We omit the definition of covering dimension since it is very complicated~(see page 54 of~\cite{RE}). We recall the following useful facts concerning compact abelian group $G$: $\dim(G)=n<\infty$ iff $G$ has a totally disconnected closed subgroup $\Delta$ such that $G/\Delta\cong \T^n$ iff there is a compact totally disconnected subgroup $N$ of $G$ and a continuous surjective homomorphism $\varphi :N\times\R^n\rightarrow G$ which has a discrete kernel~(see Theorem 8.22 and Corollary 8.26 of~\cite{HM13}). In this case, we say that $G$ is {\it finite dimensional}~(cf.~\cite[Definitions 8.23]{HM13}). Clearly, $\dim(G)=0$ iff $G$ is totally disconnected. For more details on compact abelian groups, see~\cite{HM13}.

Now we recall two equivalence relations $E_0^\omega$ and $E(M;0)$ (see~\cite[Definition 3.2]{ding12}). The equivalence relation $E_0^\omega$ on $2^{\omega\times\omega}$ defined by
$$x E_0^\omega y\iff\forall k\,\exists m\,\forall n\geq m\,(x(n,k)=y(n,k)).$$
Fix a metric space $M$. The equivalence relation $E(M;0)$ on $M^\omega$ defined by
$$x E(M;0) y\iff\lim\limits_n d(x(n),y(n))=0.$$
From the above discussions, we can establish the following theorem.

\begin{theorem}\label{n-dim}
Let $G,H$ be locally compact abelian Polish groups.
 \begin{enumerate}
    \item If $G$ is non-archimedean, then $E(G)\leq_B E_0^\omega$.
    \item If $G$ is not non-archimedean, then $E(\R)\leq_B E(G)$.
    \item If $G$ is not non-archimedean and $G_0$ is open, then $E(G)\sim_B E(G_0)$.
    \item If $n$ is a positive integer, then $E(\T^n)\leq_B E(G)$ iff $\T^n$ embeds in $G$.
    \item If $n$ is a positive integer and $G$ is compact, then $G$ is $n$-dimensional iff $E(\R^n)<_B E(G)\leq_B E(\T^n)$.
  \end{enumerate}
\end{theorem}

\begin{proof}
(1) It follows from~\cite[Theorem 3.5.(3)]{DZ}.

(2) Note that $G$ is not totally disconnected (cf.~\cite[Theorem 1.34]{HM13}), so $G_0$ contains at least two points. We have $G_0\cong\R^n\times K$, where $K$ is a compact connected abelian group. If $n>0$, it is trivial that $E(\R)\le_BE(G)$. By Propostion~\ref{closed subgroup}, $E(K)\le_BE(G_0)\le_BE(G)$. Thus we may assume that $G$ is compact connected and $G\subseteq\T^\omega$. Note that there is a continuous homomorphism $f:$ $\R\rightarrow G$ such that $\overline{f(\R)}=G$. For $g\in G\subseteq\T^\omega$ and $p\in\omega$, let $\phi_p(g)=g(p)$. Since $G$ contains at least two points, we can find ${p_0}\in\omega$ such that $\phi_{p_0}(f(\R))\ne\{1_\T\}$, so $\phi_{p_0}(f(\R))=\T$. By~\cite[Corollary 8.24]{HM13}, the interval $[0,1]$ embeds in $G$. Then by Theorem~\ref{[0,1-embed]}, we have $E(\R)\le_BE(G)$.

(3) By~\cite[\S 24.45]{HAR}, we have $G\cong G_0\times G/G_0$. Since $G_0$ is open, $G/G_0$ is countable and discrete. By~\cite[Corollary 3.6]{DZ}, this implies that $E(G_0\times G/G_0)\sim_B E(G_0)$ and thus $E(G)\sim_B E(G_0)$.

(4) The ``if'' part follows Proposition~\ref{closed subgroup}.
Assume that $E(\T^n)\le_BE(G)$. By Theorem~\ref{lca} and~\cite[Corollary 2.3.4]{gao}, there is a closed subgroup $\Delta$ of $\T^n$ such that the group $\T^n/\Delta$ can be embedded in $G$, where $\Delta$ is non-archimedean. It is obvious that $\T^n/\Delta$ is a locally connected, connected and compact abelian Polish group. By~\cite[Proposition 8.17]{Ar}, $\T^n/\Delta\cong\T^n$.

(5) If $n=\dim(G)$, then we have $(N\times\R^n)/\Delta_1\cong G$ and $G/\Delta_2\cong \T^n$, where $N,\Delta_1$, and $\Delta_2$ are totally disconnected, and hence are non-archimedean. Then Proposition~\ref{closed subgroup} and Lemma~\ref{[0,1-embed]} imply that
$$E(\R^n)\le_BE(N\times\R^n)\leq_B E(G)\leq_B E(\T^n).$$
So we only need to show that $E(G)\nleq_B E(\R^n)$. To see this, assume toward a contradiction that $E(G)\leq_B E(\R^n)$. By Theorem \ref{lca}, there exists a continuous homomorphism $S:G_0\rightarrow\R^n$ such that $\ker(S)$ is non-archimedean.  Note that $\R^n$ has no nontrivial compact connected subgroup. So this implies that $S(G_0)=\{0\}$, contradicting that $\ker(S)$ is non-archimedean.

On the other hand, suppose $E(\R^n)<_B E(G)\leq_B E(\T^n)$. Let $m=\dim(G)$. By (1) we have $m>0$. Assume for contradiction that $m=\infty$, then there exists a continuous homomorphism $S:G_0\rightarrow\T^n$ such that $\ker(S)$ is non-archimedean. Then we have $\dim(G_0/\ker(S))=\infty$, and hence $[0,1]^\omega$ embeds into $G_0/\ker(S)$ (cf.~\cite[Corollary 8.24]{HM13}). By~\cite[Corollary 2.3.4]{gao}, $S$ induce an embedding from $G_0/\ker(S)$ to $\T^n$. So $[0,1]^\omega$ embeds into $\T^n$, contradicting that $n$ is finite. Therefore, we have $0<m<\infty$, and hence $E(\R^m)<_BE(G)\le_BE(\T^m)$. Then~\cite[Theorem 6.19]{DZ} gives $m=n$.
\end{proof}

\begin{remark}
Let $G$ and $H$ be two locally compact abelian Polish groups. Suppose that $G_0$ is an open subgroup of $G$, and that $G_0$ is compact or $G_0\cong\R$. Then theorems~\ref{lca},~\ref{cc}, and~\ref{n-dim}.(2)--(3) imply that $E(G)\leq_B E(H)$ iff there is a continuous homomorphism $S:G_0\rightarrow H_0$ such that $\ker(S)$ is non-archimedean. This generalizes Rigid Theorem, i.e., Theorem~\ref{cc}. We don't know whether this can be generalized to all locally compact abelian Polish groups.
\end{remark}

\begin{question}
Does the converse of Theorem~\ref{lca} hold for all locally compact abelian Polish groups?
\end{question}

\begin{question}
Let $G$ be a locally compact abelian Polish group. If $G$ is not non-archimedean, does $E(G)\sim_B E(G_0)$?
\end{question}

\section{$P$-adic solenoids}

Let $P=(P(0),P(1),\dots)$ be a sequence of integers greater than $1$. Recall that the {\it $P$-{adic solenoid}} $S_P$ is defined by
$$S_P=\{g\in\T^\omega:\forall l\,(g(l)=g(l+1)^{P(l)})\}.$$
In particular, if for each $i$, $P(i)$ is a prime number, then the $P$-{adic solenoid} is denoted by $\Sigma_P$~(cf.~\cite{Gu}). Let $\mathcal{P}$ denote the set of all primes. The group $S_P$ is topologically isomorphic to $\Sigma_{P'}$ for some $P'\in\mathcal{P}^\omega$ satisfying that $P(l)=P'(i_l)\cdots P'(i_{l+1}-1)$, where $0=i_0<i_1<\cdots<i_l<\cdots$. For example, we have $S_{(4,6,8,9,\dots,9,\dots)}\cong\Sigma_{(2,2,2,3,2,2,2,3,3,\dots,3,3,\dots)}$.

It is well known that, the group $\Sigma_P$ is a compact connected abelian group which is neither locally connected~(cf.~\cite{Gu}), nor arcwise connected~(see~\cite[Theorem 8.27]{Ar}). Every nontrivial proper closed subgroup $H$ of a $P$-{adic solenoid} is totally disconnected~(cf.~\cite[Proposition 2.7]{BK}), and thus $H$ is non-archimedean. Clearly, $\Sigma_P$ is an $1$-dimensional and metrizable group.

Denote $\Omega=\{\R,\T,\Sigma_P:P\in\mathcal{P}^\omega\}$.

\begin{lemma}\label{product}
Let $m,n\in\N^+$ and let $G_1, G_2,\dots,G_m,H_1,H_2\dots,H_n\in\Omega$. Then the following are equivalent:
\begin{enumerate}
  \item [(1)]$E(G_1\times G_2\times\dots\times G_m)\leq_B E(H_1\times H_2\times\dots \times H_n)$.
  \item [(2)]There is a injective map $\theta^*:\{1,2,\dots,m\}\rightarrow\{1,2,\dots,n\}$ such that $E(G_i)\leq_B E(H_{\theta^*(i)})$ for $1\leq i\le m$.
\end{enumerate}
In particular, $E(G_1^m)\leq_B E(H_1^n)$ iff $m\le n$ and $E(G_1)\leq_B E(H_1)$.
\end{lemma}

\begin{proof}
$(2)\Rightarrow(1)$ is obvious. We only prove $(1)\Rightarrow(2)$.

Denote $G=G_1\times G_2\times\dots\times G_m$ and $H=H_1\times H_2\times\dots\times H_n$. For $1\leq i\le m$, let $e^i$ be the canonical injection of $G_i$ into $G_1 \times\cdots\times G_m$, i.e., $e^i(g)=(1_{G_1},\dots,1_{G_{i-1}},g,1_{G_{i+1}},\dots,1_{G_m})$.

Suppose $E(G)\le_B E(H)$. Since $G$ and $H$ are both connected, by Theorem~\ref{lca}, there is a continuous homomorphism $S:G\rightarrow H$ such that $\ker(S)$ is non-archimedean. For each $1\leq j\le n$, let $\pi_{j}$ be the canonical projection from $H$ onto $H_j$.

Note that, except for $\R$, all groups in $\Omega$ are compact. By rearranging, we may assume that there is an $i_0\le m$ such that, $G_i$ is compact for $1\le i\le i_0$, and $G_i=\R$ for $i_0<i\le m$.


For any $1\leq i\le i_0$, since $\ker(S)$ is non-archimedean, there exists $j$ satisfying that $\pi_{j}(S(e^i(G_i)))\not=\{1_{H_j}\}$. Note that $H_j$ has no nontrivial proper connected compact subgroup. It follows that $\pi_{j}(S(e^i(G_i)))=H_j$. Now we construct a bipartite graph $G[X,Y]$ as follows. Let $X=\{G_1,G_2,\dots,G_{i_0}\}$,
$$Y=\{H_j:\exists i\,(1\le i\le i_0\mbox{ and }\pi_{j}(S(e^i(G_i)))=H_j)\}.$$
For $G_i\in X$ and $H_j\in Y$, we put an edge between $G_i$ and $H_j$ if $\pi_{j}(S(e^i(G_i)))=H_j$. Given $K\subseteq X$, we denote the set of all neighbors of the vertices in $K$ by $N(K)$.

Next we show that $\left|N(K)\right|\geq\left|K\right|$ for all $K\subseteq X$. Given $K\subseteq X$, denote
$$G^{K}=\{x\in G: x(i)=1_{G_i}\mbox{ for all } G_i\notin K\},$$
$$H^{N(K)}=\{z\in H: z(j)=1_{H_j}\mbox{ for all } H_j\notin N(K)\}.$$
Then the restriction of $S$ on $G^{K}$ is a continuous homomorphism to $H^{N(K)}$.
By Theorem~\ref{cc}, $E(G^{K})\le_B E(H^{N(K)})$. Again by Theorem~\ref{n-dim}.(5), this implies $E(\R^{\left|K\right|})\leq_B E(\T^{\left|N(K)\right|})$. Then~\cite[Theorem 6.19]{DZ} gives $\left|N(K)\right|\geq\left|K\right|$.

By Hall's theorem (cf.~\cite[Theorem 16.4]{BM}), there is a injective map $\theta^*:\{1,2,\dots,i_0\}\rightarrow\{1,2,\dots,n\}$ such that $\pi_{\theta^*(i)}(S(e^i(G_i)))=H_{\theta^*(i)}$. Since every proper closed subgroup of $G_i$ is non-archimedean, from Theorem~\ref{cc}, we have $E(G_i)\le_BE(H_{\theta^*(i)})$.

In the end, since $\dim(G)=m$ and $\dim(H)=n$, by Theorem~\ref{n-dim}.(5), we have $E(\R^m)\le_BE(\T^n)$. So $m\le n$. Since $E(\R)\le_BE(H_j)$ for each $j$, we can trivially extend $\theta^*$ to an injection from $\{1,2,\dots,m\}$ to $\{1,2,\dots,n\}$ such that $E(G_i)\le_BE(H_{\theta^*(i)})$ for each $i$.
\end{proof}

Let $P$ and $Q$ be in $\mathcal{P}^\omega$. We write $Q\preceq P$ provided there is a co-finite subset $A$ of $\omega$ and an injection $f:A\to\omega$ such that $Q(n)=P(f(n))$ for each $n\in A$~(for more details, see~\cite{Gu,AG,Pr}).

\begin{lemma}[folklore]
Let $P$ and $Q$ be in $\mathcal{P}^\omega$. Then the following are equivalent:
\begin{enumerate}\label{equivalence}
\item[(1)] There is a nonzero continuous homomorphism  $f:\Sigma_P\rightarrow\Sigma_Q$.
\item[(2)] There is a surjective continuous homomorphism  $g:\Sigma_P\rightarrow\Sigma_Q$.
\item[(3)] There is a surjective continuous map $h:\Sigma_P\rightarrow\Sigma_Q$.
\item[(4)] $Q\preceq P$.
\end{enumerate}
\end{lemma}

\begin{proof}
$(2)\Rightarrow(1)$ and $(2)\Rightarrow(3)$ are obvious. $(1)\Rightarrow(2)$ follows immediately from the fact that each nontrivial proper closed subgroup of a $P$-{adic solenoid} is totally disconnected. The equivalence of $(3)$ and $(4)$ follows from~\cite[Theorem 4.4]{Pr}.

It remains to show $(3)\Rightarrow(1)$. Let $h$ be a surjective continuous map from $\Sigma_P$ to $\Sigma_Q$.
Without loss of generality assume that $h(1_{\Sigma_P})=1_{\Sigma_Q}$. Then there exists a continuous homomorphism $f:\Sigma_P\rightarrow\Sigma_Q$ such that $h$ is homotopic to $f$~(cf. \cite[Corollary 2]{Sc}). Since $\Sigma_Q$ is not arcwise connected, $\ker(f)\ne\Sigma_P$.
\end{proof}

\begin{theorem}\label{solenoid}
Let $P$ and $Q$ be in $\mathcal{P}^\omega$. Then $E(\Sigma_P)\leq_B E(\Sigma_Q)$ iff $Q\preceq P$ iff there is a nonzero continuous homomorphism $f:\Sigma_P\rightarrow\Sigma_Q$.
\end{theorem}

\begin{proof}
Note that every nontrivial proper closed subgroup of $\Sigma_P$ is non-archimedean. Then this follows from Theorem~\ref{cc} and Lemma~\ref{equivalence}.
\end{proof}

Let ${\rm Fin}$ denote the set of all finite subsets of $\omega$. For $A,B\subseteq\omega$, we use $A\subseteq^* B$ to denote $A\setminus B\in{\rm Fin}$.

We prove that, for $n\in\N^+$, the partially ordered set $P(\omega)/{\rm Fin}$ can be embedded into Borel equivalence relations between $E(\R^n)$ and $E(\T^n)$.

\begin{lemma}\label{between}
Let $P$ be in $\mathcal{P}^\omega$. Then $E(\R)<_B E(\Sigma_P)<_B E(\T)$.
\end{lemma}

\begin{proof}
By Theorem~\ref{n-dim}.(5), we have that $E(\R)<_B E(\Sigma_P)\leq_B E(\T)$.

Assume toward a contradiction that $E(\T)\leq_B E(\Sigma_P)$. From Theorem~\ref{n-dim}.(4), $\T$ embeds in $\Sigma_P$. This is impossible, since $\Sigma_P$ is not arcwise connected and every proper closed subgroup of $\Sigma_P$ is non-archimedean.
\end{proof}

For $P\in\mathcal{P}^\omega$ and $\gamma\in\mathcal{P}$, we define $t^{P}(\gamma)\in\omega\cup\{\omega\}$ as
$$t^{P}(\gamma)=\left\{\begin{array}{ll} \omega, & \exists^\infty j\in\omega\,(P(j)=\gamma),\cr \left|\{j:P(j)=\gamma\}\right|,
 & \mbox{otherwise.}\end{array}\right.$$
Given $P,Q\in\mathcal{P}^\omega$, denote
$$D(P,Q)=\{\gamma\in\mathcal{P}:t^P(\gamma)<t^Q(\gamma)\}.$$
From the definition of $Q\preceq P$, we can easily see that
$$E(\Sigma_P)\le_B E(\Sigma_Q)\iff Q\preceq P\iff\sum_{\gamma\in D(P,Q)}(t^Q(\gamma)-t^P(\gamma))\mbox{ is finite.}$$

\begin{lemma}\label{order}
Let $P,Q\in\mathcal{P}^\omega$ with $E(\Sigma_Q)\leq_B E(\Sigma_P)$.
Suppose that $D(P,Q)$ is infinite. Then for $A\subseteq\omega$, there is a group $\Sigma_{P_A}$ such that $E(\Sigma_Q)<_B E(\Sigma_{P_A})<_B E(\Sigma_P)$ and for $A,B\subseteq\omega$, we have
$$A\subseteq^*B\Longleftrightarrow E(\Sigma_{P_A})\leq_B E(\Sigma_{P_B}).$$
\end{lemma}

\begin{proof}
Enumerate $D(P,Q)$ as $d_0<d_1<d_2<\cdots$. Let $P_0^*\in\mathcal{P}^\omega$ such that $P_0^*(i)=d_{3i}$ for all $i\in\omega$.

For $L,M\in\mathcal{P}^\omega$, we define an element $L\oplus M\in\mathcal{P}^\omega$ as
$$(L\oplus M)(n)=\left\{\begin{array}{ll}L(k), & n=2k,\cr M(k), & n=2k+1.\end{array}\right.$$
It is clear that
$$t^{L\oplus M}(\gamma)=\left\{\begin{array}{ll}\omega,& t^L(\gamma)=\omega\mbox{ or }t^M(\gamma)=\omega,\cr
t^L(\gamma)+t^M(\gamma),& \mbox{otherwise.}\end{array}\right.$$

Given a set $A\subseteq\omega$, define $P_A\in\mathcal{P}^\omega$ as follows. If $\omega\setminus A$ is finite, put $P_A=P_0^*\oplus P$. Then
$$t^{P_A}(\gamma)=\left\{\begin{array}{ll}t^P(\gamma)+1, & \gamma=d_{3i},i\in\omega,\cr t^P(\gamma), & \mbox{otherwise.}\end{array}\right.$$
If $\omega\setminus A$ is infinite, enumerate it as $a_0<a_1<a_2<\cdots$. Define $P_A^*\in\mathcal{P}^\omega$ as $P_A^*(j)=d_{1+3a_j}$ for $j\in\omega$, and put $P_A=P_A^*\oplus(P_0^*\oplus P)$. Then
$$t^{P_A}(\gamma)=\left\{\begin{array}{ll}t^P(\gamma)+1, & (\gamma=d_{3i},i\in\omega)\mbox{ or }(\gamma=d_{1+3a},a\in(\omega\setminus A)),\cr
t^P(\gamma), & \mbox{otherwise.}\end{array}\right.$$

Next we show that $E(\Sigma_Q)<_B E(\Sigma_{P_A})<_B E(\Sigma_P)$ for all $A\subseteq\omega$.

First, since $t^P(\gamma)\le t^{P_A}(\gamma)$ for all $\gamma\in\mathcal P$, we have $D(P_A,P)=\emptyset$. So $P\preceq P_A$, and hence $E(\Sigma_{P_A})\le_B E(\Sigma_P)$.

Since $E(\Sigma_Q)\leq_B E(\Sigma_P)$, by Theorem~\ref{solenoid}, we have $P\preceq Q$, and hence
$$\sum_{\gamma\in D(Q,P)}(t^P(\gamma)-t^Q(\gamma))\mbox{ is finite.}$$
Note that $t^{P_A}(\gamma)=t^P(\gamma)+1$ only occurs when $t^Q(\gamma)>t^P(\gamma)$ holds, in which case we always have $\gamma\notin D(Q,P_A)$. So we have $D(Q,P_A)=D(Q,P)$ and $t^{P_A}(\gamma)=t^P(\gamma)$ for all $\gamma\in D(Q,P_A)$. This gives $E(\Sigma_Q)\leq_B E(\Sigma_{P_A})$.

Since $d_{3i}\in D(P,P_A)$ for $i\in\omega$, $D(P,P_A)$ is infinite, so $E(\Sigma_P)\not\le_B E(\Sigma_{P_A})$. Similarly, since $t^{P_A}(d_{2+3i})=t^P(d_{2+3i})<t^Q(d_{2+3i})$, we have $d_{2+3i}\in D(P_A,Q)$ for $i\in\omega$, so $E(\Sigma_{P_A})\not\le_B E(\Sigma_Q)$.

Given $A,B\subseteq\omega$, note that $A\subseteq^* B$ iff $(\omega\setminus B)\setminus(\omega\setminus A)=(A\setminus B)$ is finite. We will check that $A\subseteq^* B$ iff $P_B\preceq P_A$. We consider four cases as follows. (1) If both $\omega\setminus A$ and $\omega\setminus B$ are finite, then we have $A\subseteq^* B$ and $P_A=P_B=P_0^*\oplus P$. (2) If $\omega\setminus A$ is infinite and $\omega\setminus B$ is finite, then we have $A\subseteq^* B$ and $P_B=P_0^*\oplus P\preceq P_A^*\oplus(P_0^*\oplus P)=P_A$, since $t^{P_B}(\gamma)\le t^{P_A}(\gamma)$ for all $\gamma\in\mathcal P$. (3) If $\omega\setminus A$ is finite and $\omega\setminus B$ is infinite, then $A\not\subseteq^* B$ and $P_B=P_B^*\oplus(P_0^*\oplus P)\not\preceq P_0^*\oplus P=P_A$, since $t^{P_A}(d_{1+3b})<t^{P_B}(d_{1+3b})$ for $b\in(\omega\setminus B)$. (4) If both $\omega\setminus A$ and $\omega\setminus B$ are infinite, then $t^{P_A}(\gamma)<t^{P_B}(\gamma)$ iff $\gamma=d_{1+3b}$ for some $b\in(\omega\setminus B)\setminus(\omega\setminus A)=(A\setminus B)$. Moreover, $t^{P_B}(d_{1+3b})=t^P(d_{1+3b})+1=t^{P_A}(d_{1+3b})+1$ for all $b\in(A\setminus B)$. So
$$\sum_{\gamma\in D(P_A,P_B)}(t^{P_B}(\gamma)-t^{P_A}(\gamma))=|A\setminus B|,$$
and hence $A\subseteq^* B$ iff $P_B\preceq P_A$.

Again by Theorem~\ref{solenoid}, we have $A\subseteq^* B$ iff  $E(\Sigma_{P_A})\le_B E(\Sigma_{P_B})$.
\end{proof}

\begin{theorem}
Let $n\in\mathbb{N}^+$. Then for $A\subseteq\omega$, there is an $n$-dimensional compact connected abelian Polish group $G_A$ such that $E(\R^n)<_B E(G_A)<_B E(\T^n)$ and for $A,B\subseteq\omega$, we have
$$A\subseteq^*B\Longleftrightarrow E(G_A)\leq_B E(G_B).$$
\end{theorem}

\begin{proof}
It follows from Theorem~\ref{n-dim}.(5), lemmas~\ref{product}, \ref{between}, and \ref{order}.
\end{proof}

\section{Dual groups}

Let $G$ and $H$ be two abelian topological groups. Denote the class of all continuous homomorphisms of $G$ to $H$ by ${\rm Hom}(G,H)$, which is an abelian group under pointwise addition. We always equip Hom$(G,H)$ with compact-open topology. The abelian topological group Hom$(G,\T)$ is called the {\it dual group} of $G$, denoted by $\widehat{G}$~(cf.~\cite[Definition 7.4]{HM13}).

Let $(A,+)$ be an abelian group whose identity element denoted by $0_A$. We say that $(A,+)$ is a {\it torsion group} if each element of $A$ is finite order. We say that $(A,+)$ is {\it torsion-free} if $n\cdot g\neq 0_A$ for all $g\in A$ with $g\neq 0_A$ and $n\in\N^+$. A subset $X$ of $A$ is {\it free} if any equation $\sum_{x\in X} n_x\cdot x=0_A$ implies $n_x=0$ for all $x\in X$. The {\it torsion-free rank} of $A$, written rank$(A)$, is the cardinal number~(uniquely determined) of any maximal free subset of $A$.

Each Hausdorff locally compact abelian group $G$ is reflexive, thus it is topologically isomorphic to the double dual group $\widehat{\widehat{G}}$~(cf.~\cite[Theorem 7.63]{HM13}). A Hausdorff locally compact abelian group is compact and metrizable iff its dual group is a countable discrete group~(cf. Proposition 7.5.(i) and Theorem 8.45 of~\cite{HM13}). Let $G$ be a Hausdorff compact abelian group, then $G$ is connected iff $\widehat{G}$ is torsion-free; and $G$ is totally disconnected iff $\widehat G$ is torsion~(cf.~\cite[Corollary 8.5]{HM13}). For any finite dimensional compact abelian Polish group $G$, the covering dimension of $A$ is equal to rank$(\widehat{G})$~(cf. Lemma 8.13 and Corollary 8.26 of~\cite{HM13}).

If $H$ is a subset of an abelian topological group $G$, then the subgroup
$$H^\perp=\{\gamma\in\widehat{G}:\forall x\in H\,(\gamma(x)=1_\T)\}$$
is called the \emph{annihilator} of $H$ in $\widehat{G}$ (cf.~\cite[Definition 7.12]{HM13}).

Now we focus on compact connected abelian Polish groups.

\begin{theorem}[Dual Rigid Theorem]\label{dual rigid}
Let $G$ be a compact connected abelian Polish group and $H$ a locally compact abelian Polish group. Then $E(G)\leq_B E(H)$ iff there is a continuous homomorphism $S^*:\widehat{H}\to\widehat{G}$ such that $\widehat{G}/{\rm im}(S^*)$ is a torsion group.
\end{theorem}

\begin{proof}
$(\Rightarrow)$. We assume that $E(G)\leq_B E(H)$. By Theorem~\ref{lca}, there is a continuous homomorphism $S:G\rightarrow H$ such that $\ker(S)$ is non-archimedean. This implies that there is a homomorphism $S^*$ from $\widehat{H}$ to $\widehat{G}$ such that $\ker(S)\cong {{\rm im}(S^*)}^\perp$  ~(cf.~\cite[P.22 and P.23(a)]{Ar}). By~\cite[Lemma 7.13(\romannumeral2)]{HM13}, we have that $\ker(S)\cong(\widehat{G}/{\rm im}(S^*))^{\widehat{}}$, and hence $\widehat{\ker(S)}\cong\widehat{G}/{\rm im}(S^*)$. Since $\ker(S)$ is non-archimedean, thus is totally disconnected, so $\widehat{G}/{\rm im}(S^*)$ is a torsion group.

$(\Leftarrow)$. Since $G\cong{\widehat{\widehat{G}}}$ and $H\cong\widehat{\widehat{H}}$, we can define $S:G\to H$ via $(S^*)^*:{\widehat{\widehat G}}\to{\widehat{\widehat H}}$ (cf.~\cite[(24.41)]{HAR}). Then the similar arguments as the preceding paragraph give the desired result.
\end{proof}

\begin{corollary}\label{dual group}
Let $G$ be a compact connected abelian Polish group and $H$ a locally compact abelian Polish group. If $E(G)\leq_B E(H)$, then there is a nonzero continuous homomorphism $S^*:\widehat{H}\to\widehat{G}$.
\end{corollary}

\begin{proof}
It follows from Theorem~\ref{dual rigid} and that $\widehat{G}$ is non-torsion.
\end{proof}

\begin{example}
$\widehat{\mathbb{T}}\cong \Z$~(cf.~\cite[Examples 23.27(a)]{HAR}). Fix a $P\in\mathcal{P}^\omega$, then $\widehat{\Sigma_P}\cong\left\{\frac{m}{P(0)P(1)\dots P(n)}:m\in\Z,n\in\N\right\}$~(see~\cite[25.3]{HAR}). In view of Corollary~\ref{dual group}, we get $E(\T)\nleq_B E(\Sigma_P)$ again.
\end{example}

Recall that $\widehat{\Q}\cong S_{(2,3,4,5,6,\dots)}$~(see~\cite[25.4]{HAR}). We have the following.

\begin{corollary}
Let $G$ be a $n$-dimensional compact abelian Polish group with $n\in\N^+$. Then $E((\widehat{\Q})^n)\leq_B E(G)$.
\end{corollary}

\begin{proof}
By~\cite[Theorem 8.22.(4)]{HM13}, $G_0\cong(\widehat{\Q})^n/\Delta$, where $\Delta$ is a compact totally disconnected subgroup of $(\widehat{\Q})^n$. Again by Theorem~\ref{cc}, this means that $E((\widehat{\Q})^n)\leq_B E(G_0)$, and thus $E((\widehat{\Q})^n)\leq_B E(G)$.
\end{proof}

From the arguments above, if $\Gamma$ is a countable discrete torsion-free abelian group, then $\widehat{\Gamma}$ is a compact connected abelian Polish group.

\begin{remark}
Let $G$ be a compact connected Polish group with $E(\R^n)\le_BE(G)\le_BE(\T^n)$ for some $n>0$. By Theorem~\ref{n-dim}.(5), $\dim(G)=n$, so ${\rm rank}(\widehat G)=n$. Thus $\widehat G$ is isomorphic to a subgroup of $\mathbb Q^n$ (cf.~\cite[Exercise 13.4.3]{gao}). In particular, if $n=1$, we have either $G\cong\T$ or there exists a $P\in\mathcal P^\omega$ such that $G\cong\Sigma_P$.
\end{remark}
 The following proposition shows that, if $n>1$, the structure of $G$ can be more complicated.

\begin{proposition}
There is a $2$-dimensional compact connected Polish group $G$ such that $E(G)\nleq_B E(\Sigma_{P_0}\times\Sigma_{P_1}\times\dots\times\Sigma_{P_n})$ for $n\in\N$ and each $P_i\in\mathcal{P}^\omega$. Moreover, if $\left|\{i\in\omega:P(i)=2\}\right|<\infty$, then $E(\Sigma_{P})\nleq_B E(G)$.

\end{proposition}
\begin{proof}
Pontryagin has constructed a countable torsion-free abelian group $\Gamma\subseteq\Q^2$ whose rank is two~(cf.~\cite[Example 2]{Pon}). Then $\widehat{\Gamma}$ is a $2$-dimensional compact connected abelian Polish group. The group $\Gamma$ defined by its generators $\eta,\xi_i,(i=0,1,2\dots)$ and relations,
$$ 2^{k_{i+1}}\xi_{i+1}=\xi_i+\eta,\leqno{(**)}$$
where $i\in\omega$ and $k_i\in\N^+$ such that $\sup\{k_i:i\in\omega\}=\infty$.

Put $G=\widehat{\Gamma}$. We claim that $E(G)\nleq_B E(\Sigma_{P_0}\times\Sigma_{P_1}\times\dots\times\Sigma_{P_n})$. Otherwise, by Corollary~\ref{dual group} and~\cite[Theorem 23.18]{HAR}, there exists $i\leq n$ such that there is a nonzero continuous homomorphism $f$ from $\widehat{\Sigma_{P_i}}$ to $\widehat{G}$. Note that for any $a\in$ $\widehat{\Sigma_{P_i}}$, there are infinitely many positive integers $n$ such that the equation $nx=a$ has a solution. But any element in $\Gamma$ does not admit such property. This implies that $f(\widehat{\Sigma_{P_i}})=\{1_\Gamma\}$ contradicting that $f$ is a nonzero homomorphism.

Now assume that $E(\Sigma_{P})\leq_B E(G)$ for some $P\in\mathcal P^\omega$. We show that $\{i\in\omega:P(i)=2\}$ is infinite. By Corollary~\ref{dual group}, there is a nonzero homomorphism $f$ from $\widehat{G}$ to $\widehat{\Sigma_{P}}$. Without loss of generality we may assume $\widehat{G}=\Gamma$ and $\widehat{\Sigma_P}=\left\{\frac{m}{P(0)P(1)\dots P(n)}:m\in\Z,n\in\N\right\}\subseteq\Q$. From $(**)$, a straightforward calculation shows that
$$2^{k_1+k_2+\dots+k_i}\xi_i=\xi_0+\eta(1+2^{k_1}+2^{k_1+k_2}+\dots+2^{k_1+k_2+\dots+k_{i-1}}).$$
So we have
$$2^{k_1+k_2+\dots+k_i}f(\xi_i)=f(\xi_0)+f(\eta)(1+2^{k_1}+2^{k_1+k_2}+\dots+2^{k_1+k_2+\dots+k_{i-1}}).$$
Note that $\lim_i2^{-(k_1+k_2+\dots+k_i)}f(\xi_0)=0$ and
$$\frac{1+2^{k_1}+2^{k_1+k_2}+\dots+2^{k_1+k_2+\dots+k_{i-1}}}{2^{k_1+k_2+\dots+k_i}}f(\eta)\le\frac{f(\eta)}{2^{k_i-1}}\to 0\quad(i\to\infty).$$
This implies that $\lim_if(\xi_i)=0$.

Let $f(\xi_0)=a/b$ and $f(\eta)=c/d$ for some integers $a,b,c,d$ with $c,d>0$. Note that $2^{k_{i+1}}f(\xi_{i+1})=f(\xi_i)+f(\eta)$. Since $f$ is a nonzero homomorphism, there can be at most one $f(\xi_i)=0$. For large enough $i$, we have $f(\xi_i)\ne 0$. So there exist integers $m_i,m'_i,c',d',l_i$ with $m_i,m'_i\ne 0$ and $c',d'>0$ such that
$$f(\xi_i)=\frac{m_i}{2^{k_1+k_2+\dots+k_i}cd}=\frac{m'_i}{2^{l_i}c'd'},$$
where $m'_i$ and $2^{l_i}c'd'$ are coprime and $c'|c$, $d'|d$. It follows that
$$|f(\xi_i)|\ge\frac{1}{2^{l_i}c'd'}\ge\frac{1}{2^{l_i}cd}\to 0\quad(i\to\infty).$$
So $l_i\to\infty$ as $i\to\infty$, and hence $\{i\in\omega:P(i)=2\}$ is infinite.
\end{proof}

\subsection*{Acknowledgements}
The authors would like to thank the anonymous referee for helpful comments and suggestions.

\end{document}